\documentclass[12pt]{article}      % Comments after  % are ignored^M
\usepackage{amsmath,amssymb,amsthm, amscd, graphicx, verbatim, setspace} % Typical maths resource packages^M

\theoremstyle{definition}
\newtheorem{mydef}{Definition}

\theoremstyle{plain}
\newtheorem{thm}{Theorem}[section]
\newtheorem*{thm*}{Theorem}
\newtheorem{corollary}{Corollary}[section]
\newtheorem{lemma}{Lemma}[section]

\title{Bounding extremal functions of forbidden $0-1$ matrices using $(r,s)$-formations}

\date{}
\author{Jesse Geneson\\
\small\tt geneson@gmail.com\\
Meghal Gupta\\
\small\tt meghal.gupta@gmail.com
}
\begin{document}
\maketitle

\begin{abstract}
First, we prove tight bounds of $n 2^{\frac{1}{(t-2)!}\alpha(n)^{t-2} \pm O(\alpha(n)^{t-3})}$ on the extremal function of the forbidden pair of ordered sequences $(1 2 3 \ldots k)^t$ and $(k \ldots 3 2 1)^t$ using bounds on a class of sequences called $(r,s)$-formations. Then, we show how an analogous method can be used to derive similar bounds on the extremal functions of forbidden pairs of $0-1$ matrices consisting of horizontal concatenations of identical identity matrices and their horizontal reflections. 
\end{abstract}

\section{Introduction}
A generalized Davenport-Schinzel sequence that avoids the sequence $u$ is a sequence that has no subsequence isomorphic to $u$. Generalized Davenport-Schinzel sequences have a wide range of applications in mathematics. They have been used to bound the complexity of the lower envelope of a set of polynomial curves in the plane \cite{lowerenv}, as well as the maximum number of edges in $k$-quasiplanar graphs on $n$ vertices with no pair of edges intersecting in more than $O(1)$ points \cite{quasi, kquasi}. 

Related extremal problems have also been studied for 0-1 matrices. Extremal problems for 0-1 matrices correspond to extremal problems about ordered bipartite graphs. The maximum number of ones in a 0-1 matrix avoiding a 0-1 matrix pattern $P$ is the same as the maximum number of edges in an ordered bipartite graph with vertex sets $\{1,\ldots,n\}$ and $\{n+1,\ldots,2n\}$ avoiding a certain ordered bipartite graph pattern $P'$ \cite{graphs}. Extremal functions of 0-1 matrices have also been used to find the shortest rectilinear path in a grid with obstacles \cite{rectilinear}, which has applications to robot navigation. F\"{u}redi also used the extremal function to determine an upper bound on the maximum number of unit distances in a convex $n$-gon \cite{ngon}.

\subsection{Definitions and Results}
\begin{mydef} For a sequence $S$ of letters, let $|S|$ denote the length of $S$, and $||S||$ denote the number of distinct letters in $S$. The sequence $S$ $contains$ another sequence $u$ if some subsequence of $S$ is isomorphic to $u$. $S$ $avoids$ $u$ if it does not $contain$ $u$. The unordered extremal function, which we denote $ex_u(v,n)$, is the maximum length of a $||v||-sparse$ sequence with $n$ distinct letters that $avoids$ $v$. For a family $F$ of sequences with the same number of distinct letters, we define $ex_u(F,n)$ to be the maximum length of such a sequence that $avoids$ all sequences in $F$.
\end{mydef}

\begin{mydef}An ordered sequence is a set of symbols with an ordering; for convenience, we will use numbers. A sequence $S$ $order$-$contains$ another sequence $u$ if some subsequence of $S$ is isomorphic to $u$, by means of an isomorphism that preserves order. For convenience, we simply use $contain$ and $avoid$ if it is clear we are referring to ordered sequences. The ordered extremal function $ex_o(P,n)$ is the maximum length of a $||u||-sparse$ sequence with $n$ distinct letters that $order$-$avoids$ $u$. For a $family$ $F$, $ex_o(F,n)$ is the maximum length of such a sequence that $order$-$avoids$ all sequences in $F$.
\end{mydef}

\begin{mydef}
An $(r,s)$-formation is a concatenation of $s$ permutations on the same $r$ symbols. Let $\zeta_{r,s}(n)$ be $ex_u(F_{r,s},n)$, where $F_{r,s}$ is the set of all $(r,s)$-formations.
\end{mydef}

Agarwal, Sharir, Shor, and Hart proved most of the lower bounds for $\zeta_{r,s}(n)$ \cite{lowerbounds, lowerbounds2}, but we only use the upper bounds in this paper. Pettie sharpened the bounds and extended them to doubled $(r,s)$-formations \cite{pettie}. 

\begin{mydef}
A $j$-fat permutation on $r$ symbols is a sequence with $r$ distinct letters and $j$ occurrences of each letter. A $j$-tuple $(r,s)$-formation is a concatenation of $s$ $j$-fat permutations on the same $r$ symbols. Let $\Phi_{r,s}(n)$ be $ex_u(FF_{r,s},n)$, where $FF_{r,s}$ is the set of all $r$-tuple $(r,s)$-formations.
\end{mydef}

Nivasch \cite{rsbounds} and Pettie \cite{pettie} proved the upper bounds on $\zeta_{r,s}(n)$ and $\Phi_{r,s}(n)$ for $r \geq 3$.

\begin{thm} \label{zetabounds} \cite{rsbounds, pettie} 

$\zeta_{r,s}(n), \Phi_{r,s}(n) = \begin{cases} \Theta(n) &\mbox{if } s \leq 3 \\ \Theta(n \alpha(n)) &\mbox{if } s = 4 \\ \Theta(n 2^{\alpha(n)}) &\mbox{if } s = 5 \\ \Theta(n 2^{\alpha(n)^{t}(\log \alpha(n) \pm O(1))/t!}) &\mbox{if } s \geq 6, \text{$s$ even} \\ n \cdot 2^{(1/t!)\alpha (n)^t \pm O(\alpha (n)^{t-1})} & \mbox{if } s\geq 7, \text{$s$ odd}. \end{cases}$ \\

where $t=\left \lfloor \frac{s-3}{2} \right \rfloor$ and $r \geq 3$. (The O notation hides factors dependent on r and s.) 
\end{thm}

For unordered sequences, the $(r,s)$-formation upper bounds have been used to provide a generalized bound on the extremal functions of all sequences: For a sequence $u$, $ex_u(u,n) \leq \zeta_{r,s-r+1}$, where $r=||u||$ and $s=|u|$ \cite{rsbounds}.

In \cite{gpt}, a function of sequences called formation width ($fw$) was defined so that $fw(u)$ is the minimum $s$ for which there exists $r$ such that every $(r, s)$-formation contain $u$. The paper showed that $fw((a b c \ldots)^t) = 2t-1$ and $ex_{u}((a b c \ldots)^t,n) = n 2^{\frac{1}{(t-2)!}\alpha(n)^{t-2} \pm O(\alpha(n)^{t-3})}$. This result implied an improved upper bound on the maximum number of edges in $k$-quasiplanar graphs with no pair of edges intersecting in more than $O(1)$ points. 

In this paper, we apply $(r,s)$-formations to bound extremal functions of forbidden pairs of ordered sequences. The bounds on $\zeta_{r,s}(n)$ apply to ordered $(r,s)$-formations as well. Using this, in Theorem \ref{2seq} of our paper, we prove the following: Let $\sigma_1=(1 2 3 \ldots k)^t$ and $\sigma_2=(k \ldots 3 2 1)^t$. Then $ex_{o}(\{\sigma_1, \sigma_2\},n) = n 2^{\frac{1}{(t-2)!}\alpha(n)^{t-2} \pm O(\alpha(n)^{t-3})}$.

\begin{mydef}A 0-1 matrix $A$ $contains$ another $P$ if some submatrix of $A$ can be transformed into $P$ by possibly changing some ones to zeroes. $A$ $avoids$ $P$ if it does not $contain$ $P$. Denote $ex(P,n)$ as the maximum number of ones an $n \times n$ 0-1 matrix that $avoids$ $P$ can have. For a $family$ $F$, let $ex(F,n)$ be the maximum number of ones in such a matrix that $avoids$ every element of $F$.
\end{mydef}

We use an analogue of $(r,s)$-formations in $0-1$ matrices to prove a similar result for forbidden pairs of $0-1$ matrices.

\begin{mydef} For an ordered sequence $S$ define $\chi(S)$ to be the corresponding 0-1 matrix where the only 1-entry in column $c$ is in the row number in spot $c$ of $S$. For example,
\[\chi(12323)=\left( \begin{array}{ccccc}
1 & 0 & 0 & 0 & 0\\
0 & 1 & 0 & 1 & 0\\
0 & 0 & 1 & 0 & 1\end{array} \right)\]
\end{mydef}

This is an intuitive correspondence because if sequence $a$ $contains$ sequence $b$, then $\chi(a)$ will $contain$ $\chi(b)$.

\begin{mydef}
Define a permutation matrix $(r,s)$-formation as a concatenation of $s$ $r \times r$ permutation matrices, or $\chi(a)$, where $a$ is an ordered sequence $(r,s)$-formation. For a given $r,s$ let $G_{r,s}$ be all permutation matrix $(r,s)$-formations. Denote $ex(G_{r,s},n)$ as $\lambda_{r,s}(n)$.
\end{mydef}

Cibulka, Kyncl, and Pettie \cite{vc, pettie} proved bounds that imply upper bounds on the extremal function of the family of permutation matrix $(r,s)$-formations.

\begin{mydef}
Let a $B$-$fat$ $r \times r$ permutation matrix be a matrix with $r$ rows in which every row has $B$ ones. Let an $(r,s)$-$B$-fat be a matrix of $s$ horizontally concatenated $r \times r$ $B$-$fat$ permutation matrices. Let $H_{r,s}$ be the set of all $(r,s)$-$r$-fats for a given $r,s$, and let $\Gamma_{r,s}(n) = ex(H_{r,s},n)$.
\end{mydef}

Cibulka and Kyncl used bounds on $\Gamma_{r,s}(n)$ to find bounds on the maximum size of permutations with VC-dimension $k$ for a given $k$ \cite{vc}. The following bounds on $\lambda_{r,s}(n)$ follow trivially from the bounds on $\Gamma_{r,s}(n)$ in \cite{vc, pettie}, where $t=\left \lfloor \frac{s-3}{2} \right \rfloor$:

\[\lambda_{r,s}(n), \Gamma_{r,s}(n) \leq \begin{cases} O(n) &\mbox{if } s \leq 3 \\ O(n\alpha(n)^{2}) & \mbox{if } s = 4 \\ n \cdot 2^{(1/t!)\alpha(n)^t+O(\alpha(n)^{t-1})} & \mbox{if } \text{$s$ odd, $s \geq 5$} \\ n \cdot 2^{(1/t!)\alpha (n)^t \log _2 \alpha (n) + O(\alpha (n)^t)} & \mbox{if } \text{$s$ even, $s\geq 6$}. \end{cases}\]

In Theorem \ref{abcmatrix}, we bound $ex(\{A,B\},n)$ where $A$ is a concatenation of identical identity matrices, and $B$ is the horizontal reflection of $A$. The bounds are especially interesting in the case that $A$ and $B$ are both the concatenation of exactly two matrices, since in this case $ex(\{A,B\},n) = \Theta(n)$.

\section{Sequence Results}

We first note the following trivial lemma.

\begin{lemma}
Let $F_{r,s}$ be the family of ordered $(r,s)$-formations for a given $r$ and $s$. Then, $ex_o(F_{r,s},n)=ex_u(F_{r,s},n)=\zeta_{r,s}(n)$. Consequently, the bounds in Theorem \ref{zetabounds} still apply.
\end{lemma}

Any sequence or family of sequences that is guaranteed to be contained in every $(r,s)$-formation for a given $r$ and $s$ has extremal function at most $ex(F_{r,s},n)$. This is the guiding principle behind our subsequent proofs.

\begin{lemma}\label{fwop}
If $\sigma_1=(1 2 3 \ldots k)^t$ and $\sigma_2=(k \ldots 3 2 1)^t$, then $fw(\left\{ \sigma_1, \sigma_2 \right\}) = 2t-1$ for $k \geq 2$.
\end{lemma}

\begin{proof}
The lower bound follows since $fw((a b \ldots)^{t}) = 2t-1$. The upper bound follows from a result in \cite{gpt}. Let a binary $(r,s)$-formation be an $(r,s)$-formation in which every permutation is $1 \ldots r$ or $r \ldots 1$. For $\gamma$ sufficiently large, every $(\gamma,s)$-formation contains a binary $(r,s)$-formation. Thus $fw(\left\{ \sigma_1, \sigma_2 \right\}) \leq 2t-1$ by the pigeonhole principle.
\end{proof}

If $F$ denotes the family of ordered sequences that are isomorphic to $(1 \ldots k)^{t}$, then $ex_{o}(F, n) = ex_{u}((1 \ldots k)^{t}, n)$. However, $ex_{o}((1 \ldots k)^{t}, n) > ex_{u}((1 \ldots k)^{t}, n)$. 

In the next theorem, we show that only $2$ of the $k!$ permutations of $1, \ldots, k$ are necessary for the ordered sequence extremal function to be within a constant factor of the unordered sequence extremal function. Specifically, $ex_{o}(\left\{ (1 \ldots k)^{t}, (k \ldots 1)^{t} \right\}, n) = \Theta(ex_{u}((1 \ldots k)^{t}, n))$.

\begin{thm} \label{2seq}
If $\sigma_1=(1 \ldots k)^t$ and $\sigma_2=(k \ldots 1)^t$, then $ex_o(\{\sigma_1, \sigma_2\},n) = n 2^{\frac{1}{(t-2)!}\alpha(n)^{t-2} \pm O(\alpha(n)^{t-3})}$ for $k \geq 2$ and $t \geq 3$.
\end{thm}

\begin{proof}
The upper bound follows from Lemma \ref{fwop}. The lower bound follows from the lower bound on $ex_{u}((a b c \ldots)^{t},n)$.
\end{proof}

Note that the value of $r$ is not important, since in the bounds for $\zeta_{r,s}(n)$, $r$ only affects the constant factor; only the value of $s$ is significant here.

In order to derive a more general corollary, we introduce two more definitions.

\begin{mydef}
If $u$ is a sequence, then $dfw(u)$ is the minimum $s$ for which there exists $r$ such that every $r$-tuple $(r,s)$-formation contains $u$. 
\end{mydef}

Below we define reduced sequences to simplify the calculation of $dfw(u)$.

\begin{mydef}
If $u$ is a sequence or family of sequences, then $red(u)$ is the result of replacing every block of adjacent same letters in $u$ with one occurrence of the letter. 
\end{mydef}

The next lemma follows trivially from Lemma 1.2 in \cite{pettie}.

\begin{lemma}\label{ftod}
$dfw(u) = fw(red(u))$ for all sequences $u$
\end{lemma}

Observe that the last lemma can be used to generalize the bounds in Theorem \ref{2seq}.

\begin{lemma}
If $\sigma_1=(1 \ldots k)^t$ and $\sigma_2=(k \ldots 1)^t$, then let $\gamma_{i}$ for $i = 1, 2$ be obtained from $\sigma_{i}$ by replacing every letter with $j$ adjacent copies of itself. Then $dfw(\{\gamma_1, \gamma_2\}) = 2t-1$.
\end{lemma}

\begin{corollary}
If $\sigma_1=(1 \ldots k)^t$ and $\sigma_2=(k \ldots 1)^t$, then let $\gamma_{i}$ for $i = 1, 2$ be obtained from $\sigma_{i}$ by replacing every letter with $j$ adjacent copies of itself. Then $ex_o(\{\gamma_1, \gamma_2\},n) = n 2^{\frac{1}{(t-2)!}\alpha(n)^{t-2} \pm O(\alpha(n)^{t-3})}$ for $k \geq 2$ and $t \geq 3$.
\end{corollary}

Lemma \ref{ftod} can also be used to generalize a result from \cite{sfw}.

\begin{corollary}
If $v$ is a sequence that contains $a b a b a$ such that $fw(v) = 4$, then $ex_{u}(v',n) = \Theta(n \alpha(n))$ for every sequence $v'$ obtained from $v$ by replacing any letters in $v$ with multiple occurrences of the same letter.
\end{corollary}

Every sequence $v$ containing $a b a b a$ with $fw(v) = 4$ is listed in \cite{sfw}. Moreover, note that every sequence $v'$ containing $a b a b a$ with $dfw(v') = 4$ can be obtained from a sequence $v$ containing $a b a b a$ with $fw(v) = 4$ by replacing any letters in $v$ with multiple occurrences of the same letter.

\section{0-1 Matrix Results}

Next we prove the corresponding $0-1$ matrix bounds, using the bounds on $\lambda_{r,s}(n)$ combined with a function for $0-1$ matrices analogous to $fw$.

\begin{mydef}
If $M$ is a $0-1$ matrix or a family of $0-1$ matrices having no columns with multiple ones, then $mfw(M)$ is the minimum $s$ for which there exists $r$ such that every permutation matrix $(r,s)$-formation contains $M$.
\end{mydef}

As with the function $fw$ for sequences, it is convenient to define binary permutation matrix $(r,s)$-formations in order to compute $mfw$.

\begin{mydef}
A binary permutation matrix $(r,s)$-formation is a permutation matrix $(r,s)$-formation in which every permutation is either the identity matrix or its reflection.
\end{mydef}

\begin{lemma}
Every permutation matrix $((r-1)^{2^{s-1}}+1,s)$-formation contains a binary permutation matrix $(r,s)$-formation.
\end{lemma}

\begin{proof}
As in \cite{gpt}, the result can be proved by inducting on $s$ and successively applying the Erdos-Szekeres theorem.
\end{proof}

\begin{corollary}
If $A_{k,t}$ is a horizontal concatenation of $t$ $k\times k$ identity matrices and $B_{k,t}$ is its horizontal reflection, then $mfw(\left\{A_{k,t},B_{k,t}\right\}) = 2t-1$. 
\end{corollary}

\begin{proof}
The upper bound follows from the last lemma and the pigeonhole principle. The lower bound follows since a binary permutation matrix $(z,2t-2)$-formation with exactly $t-1$ identity matrices avoids both $A_{k,t}$ and $B_{k,t}$.
\end{proof}

\begin{thm} \label{abcmatrix}
If $A_{k,t}$ is a horizontal concatenation of $t$ $k\times k$ identity matrices and $B_{k,t}$ is its horizontal reflection, then $ex(\{A_{k,t},B_{k,t}\},n) \leq n \cdot 2^{(1/(t-2)!)\alpha(n)^{(t-2)} + O(\alpha(n)^{t-3})}$ for $t \geq 3$.
\end{thm}

\begin{proof}
The upper bound follows from the last corollary and the upper bounds on $\lambda_{r,s}(n)$ in the introduction. 
\end{proof}

\begin{thm}
If $A_{k}$ is a horizontal concatenation of $2$ $k\times k$ identity matrices and $B_{k}$ is its horizontal reflection, then $ex(\{A_{k},B_{k}\},n) = \Theta(n)$.
\end{thm}

\begin{proof}
The upper bound follows from the last theorem, while the lower bound follows from the fact that both $A_{k}$ and $B_{k}$ have at least two ones.
\end{proof}

As with ordered sequences, we can also generalize the last two results by using doubled formation width for $0-1$ matrices.

%\begin{mydef}
%A $j$-tuple permutation matrix with $r$ rows is obtained from an $r \times r$ permutation matrix by replacing every column with $j$ copies of itself. A $j$-tuple permutation matrix $(r,s)$-formation is obtained by concatenating $s$ $j$-tuple permutation matrices with $r$ rows.
%\end{mydef}

\begin{mydef}
If $M$ is a $0-1$ matrix or a family of $0-1$ matrices having no columns with multiple ones, then $dmfw(M)$ is the minimum $s$ for which there exists $r$ such that every $r$-fat permutation matrix $(r,s)$-formation contains $M$.
\end{mydef}

We defined reduced sequences to not have any adjacent same letters. Below is a similar definition for $0-1$ matrices.

\begin{mydef}
If $M$ is a 0-1 matrix or family of 0-1 matrices having no columns with multiple ones, then $red(M)$ is the result of replacing every block of adjacent columns in $u$ having ones in the same row with a single column having a one in the row. 
\end{mydef}

Like Lemma \ref{ftod}, the next lemma is a trivial consequence of Lemma 1.2 in \cite{pettie}.

\begin{lemma}
$dmfw(u) = mfw(red(M))$ for all $0-1$ matrices $M$ having no columns with multiple ones
\end{lemma}

As with ordered sequences, we can use the last lemma to generalize the bounds in Theorem \ref{abcmatrix}.

\begin{lemma}
If $A_{j, k,t}$ is a horizontal concatenation of $t$ $k\times k$ $j$-fat identity matrices and $B_{j, k,t}$ is its horizontal reflection, then $dmfw(\{A_{k,t},B_{k,t}\}) = 2t-1$.
\end{lemma}

\begin{corollary}
If $A_{j, k,t}$ is a horizontal concatenation of $t$ $k\times k$ $j$-fat identity matrices and $B_{j, k,t}$ is its horizontal reflection, then $ex(\{A_{j,k,t},B_{j,k,t}\},n) \leq n \cdot 2^{(1/(t-2)!)\alpha(n)^{(t-2)} + O(\alpha(n)^{t-3})}$ for $t \geq 3$.
\end{corollary}

\begin{corollary}
If $A_{j, k}$ is a horizontal concatenation of $2$ $k\times k$ $j$-fat identity matrices and $B_{j, k}$ is its horizontal reflection, then $ex(\{A_{j,k},B_{j,k}\},n) = \Theta(n)$.
\end{corollary}

\section{Open Problems}

There are many unordered sequence patterns, ordered sequence patterns, 0-1 matrix patterns, and families of patterns for which the extremal functions do not yet have tight bounds. It is likely that the upper bounds for many of these extremal functions can be improved using formation width.

Clearly formation width is not useful to bound the extremal function of a single ordered sequence, since $fw(u)$ is the length of $u$ when $u$ is an ordered sequence. The same is true for using $mfw$ to bound the extremal function of a single forbidden $0-1$ matrix.

However, $fw$ was used to derive tight bounds on numerous extremal functions of unordered sequences in \cite{gpt, sfw} and on forbidden pairs of ordered sequences in this paper. Moreover, $mfw$ was used to derive tight bounds on the extremal functions of forbidden pairs of $0-1$ matrices in this paper. It is an open problem to find more families of patterns for which $fw$ and $mfw$ provide tight upper bounds on the extremal functions.

Currently the running times of the algorithms for $fw$ and $mfw$ grow exponentially in the output, which is bounded by the length of the input. The known algorithm for $fw$ is described in both \cite{gpt} and \cite{sfw}, while the known algorithm for $mfw$ is analogous. Speeding up the computation of $fw$ and $mfw$ would enable these functions to be used for bounding the extremal functions of larger patterns than the ones that have been considered so far.


\begin{thebibliography}{7}
\bibitem{lowerbounds} P. Agarwal, M. Sharir, and P. Shor. Sharp upper and lower bounds on the length of general Davenport-Schinzel sequences. J. Combin. Theory Ser. A, 52:228-274, 1989.
\bibitem{quasi} P. K. Agarwal, B. Aronov, J. Pach, R. Pollack, and M. Sharir, Quasi-planar graphs have a linear number of edges. In Proceedings of the Symposium on Graph Drawing (GD '95), Franz-JosefBrandenburg (Ed.). Springer-Verlag, London, UK, 1995, 1-7.
\bibitem{vc} J. Cibulka and J. Kyncl. Tight bounds on the maximum size of a set of permutations with bounded VC-dimension. J. Combin. Theory Ser. A, 119(7):1461-1478, 2012.
\bibitem{lowerenv} H. Davenport and A. Schinzel. A combinatorial problem connected with
differential equations. American Journal of Mathematics, 87:684-694, 1965.
\bibitem{kquasi} J. Fox, J. Pach, and A. Suk. The number of edges in k-quasiplanar graphs. SIAM Journal of Discrete Mathematics, 27:550-561, 2013.
\bibitem{ngon} Z. Furedi, The maximum number of unit distances in a convex n-gon, J. Combin. Theory Ser. A 55 (2) (1992) 316-320.
\bibitem{gpt} J. Geneson, R. Prasad, J. Tidor, Bounding Sequence Extremal Functions with Formations, Electr. J. Comb. 21(3) P3.24 (2014)
\bibitem{sfw} J. Geneson and P. Tian. Sequences of formation width 4 and alternation length 5. CoRR abs/1502.04095 (2015)
\bibitem{lowerbounds2}  S. Hart and M. Sharir. Nonlinearity of Davenport-Schinzel sequences and of generalized path compression schemes. Combinatorica, 6(2):151-177, 1986.
\bibitem{rectilinear} J. Mitchell: Shortest rectilinear paths among obstacles, Department of Operations Research and Industrial Engineering Technical Report No. 739, Cornell University, Ithaca, New York (1987)
\bibitem{rsbounds} G. Nivasch. Improved bounds and new techniques for Davenport-Schinzel sequences and their generalizations. J. ACM, 57(3), 2010.
\bibitem{pettie} S. Pettie, Three generalizations of Davenport-Schinzel sequences, arXiv:1401.5709, 2014.
\bibitem{graphs} C. Weidert. Extremal problems in ordered graphs. Master's thesis, Simon Fraser University, British Columbia, Canada, 2009.


\end{thebibliography}
\end{document}